\documentclass[11pt]{amsart}

\usepackage[OT1]{fontenc}
\usepackage[left=3.6cm,top=3.8cm,right=3.6cm]{geometry}
\usepackage{comment}
\usepackage[utf8]{inputenc}
\usepackage{enumerate}
\usepackage{amsthm,amsmath,amssymb,mathrsfs}
\usepackage{tikz}

\usepackage{verbatim}

\usepackage{verbatim} 
\usepackage{esint}
\usepackage{xcolor}
\numberwithin{equation}{section}

\usepackage{xcolor}
\usepackage[pagebackref]{hyperref}
\hypersetup{
   colorlinks,
    linkcolor={red!60!black},
    citecolor={blue!60!black},
    urlcolor={blue!90!black}
}

\numberwithin{equation}{section}

\theoremstyle{plain}

\newtheorem{theorem}{Theorem}[section]
\newtheorem{lemma}[theorem]{Lemma}
\newtheorem{proposition}[theorem]{Proposition}

\theoremstyle{definition}
\newtheorem{definition}[theorem]{Definition}

\newtheorem{remark}[theorem]{Remark}

\theoremstyle{remark}
\numberwithin{equation}{section}

\newcommand{\rr}{\mathbb{R}}

\newcommand{\Z}{\mathbb{Z}}
\newcommand{\eps}{\varepsilon}
\newcommand{\norm}[1]{\left\lVert #1 \right\rVert}
\newcommand{\1}{\mathbf{1}}

\title{off-diagonal estimates for cube skeletons maximal operators}

\author{Andrea Olivo}
\address{Departamento de Matem\'atica,
Facultad de Ciencias Exactas y Naturales,
Universidad de Buenos Aires, Ciudad Universitaria
Pabell\'on I, Buenos Aires 1428 Capital Federal Argentina} \email{aolivo@dm.uba.ar}

\thanks{The author is partially supported by grant PIP (CONICET) 11220110101018.}

\begin{document}

\begin{abstract}
We provide off-diagonal estimates for a maximal operator arising from a geometric problem of estimating the size of certain geometric configuration of $k$-skeletons in $\mathbb{R}^n$.
 This is achieved by interpolating a weak-type endpoint estimate with the known diagonal bounds. The endpoint estimate is proved by combining a geometric result about $k$-skeletons and adapting an argument used to prove off-diagonal estimates for the circular maximal function in the plane.\\
 
 2010 \textit{Mathematics Subject Classification}. Primary: 42B25. Secondary: 43A85.

\end{abstract}


\keywords{averages over skeletons, maximal functions, off-diagonal estimates}

\maketitle

\section{Introduction}

In this work we present off-diagonal estimates for maximal operators associated to averaging over (neighborhoods
of) squares in the plane and, more generally over $k$-skeletons of cubes with arbitrary
dimension. Roughly speaking,
 the $k$-dimensional
boundary of an $n$-dimensional cube with axes-parallel sides in $\mathbb{R}^n$, for $0 \leq k < n$.

The interest in this type of operators emerges, on one hand, from a geometric problem
about the size of a set containing re-scaled and translated copies of the $k$-skeleton of the
unit cube in $\mathbb{R}^n$ around every point of a set of given size. On the other hand, it is a natural
variant of the celebrated spherical maximal operator of Bourgain-Stein.

The problem of finding
 minimal values of the size for sets in $\mathbb{R}^n$ containing $k$-skeletons centered at any point of a given set
was introduced by Keleti, Nagy and Shmerkin \cite{KNS18} for the case $n=2$ and  by Thornton
 \cite{Tho17}, for the case $n\geq 3$. 
Unlike the situation of spheres, where a set containing a sphere with center in every point of a set of positive Lebesgue measure must have positive Lebesgue measure, a set containing the $(n-1)$-skeleton of an $n$-dimensional cube with axes-parallel sides and center in every point of $\mathbb{R}^n$ can have zero Lebesgue measure. Even more, it can have Hausdorff dimension $n-1$ (see \cite{Tho17}), the same as a single $(n-1)$-skeleton. This result indicates that Hausdorff dimension does not fit well with this problem. For this reason,  
 in \cite{KNS18, Tho17}, the authors obtained estimates for other notions of fractal dimensions as for example box dimension and packing dimension. The arguments from \cite{KNS18, Tho17} are direct and do not involve any maximal operators. Regardless, in \cite{KNS18} the authors introduced the maximal operator associated with the aforementioned geometric problem.
More precisely, for each $f \in L^1_{\rm loc}(\rr^n)$, $0<\delta <1$ and $0\leq k < n \in \mathbb{N}$, the $k$-\textit{skeleton maximal operator} with width $\delta$ is defined as 
\begin{equation}\label{k-skeleton maximal}
M^{k}_{\delta}f(x) = \sup\limits_{1 \leq r \leq2} \min\limits_{j=1}^{N}  \fint _ {S_{k,\delta}^{j}(x,r)} |f(y)| \,dy,
\end{equation}
where  $S_{k,\delta}(x,r)$ is a $\delta$-neighborhood of the $k$-skeleton of a cube with center $x$ and side length $2r$, and the index $j$ enumerates its $N=2^{n-k}{n \choose k}$ faces.
Here, as usual, the symbol $\fint_E f d\mu:=\frac{1}{\mu(E)}\int_E fd\mu$ denotes the average of the function $f$ with respect to the measure $\mu$ over the set $E$.

As was pointed out in \cite{KNS18}, it turns out necessary take the minimum over all the faces of the $k$-skeleton to avoid natural and trivial results, analog to the ones for Hausdorff dimension.
 On the other hand,
 this operator cannot be bounded from $L^p$ to $L^q$ for any  finite $p,q$, for otherwise a set with a $k$-skeleton centered at
every point of $\mathbb{R}^n$ would have positive measure. Following this line, Shmerkin and the present author studied in  \cite{OS18} discretized versions of $M^k_\delta$, and proved nearly sharp $L^p$ bounds for the $k$-skeleton maximal operator. Easily one can deduce that $M^k_ \delta$ is bounded on $L^p$, if $p>1$, just by comparison with the Hardy-Littlewood maximal operator.  
Nonetheless, an interesting problem is to determine the rate at which the norm of $M^k_\delta$ increases as $\delta$ goes to zero.

\begin{theorem}\cite[Theorem 1.2]{OS18}
Given $0\le k< n$, $1\le p <\infty$ and $\eps>0$, there exist  positive constants $C'(n,k,\eps), C(n,k)$ such that
\begin{equation*}
C'(n,k,\eps)\cdot \delta^{\frac{k-n}{2np}+\eps} \leq \norm{M^{k}_{\delta}}_{L^p \to L^p} \leq C(n,k)\cdot \delta^{\frac{k-n}{2np}},
\end{equation*}
for all $\delta\in (0,1)$.
\end{theorem}
The lower bound relies on a
specific construction due in \cite{Tho17} and to obtain the upper bound there is in fact a result for a bigger
(normwise) maximal type operator which is localized on a given cube and it is linear
(see Section \ref{sec:preliminaries and notation} for more details).
From this $L^p$ bound it is possible to recover the known values for the Box counting
dimension of sets containing skeletons centered at any point of a prescribed set of
centers under the condition of having full Box dimension (see \cite[Corollary 3.5]{OS18}).


Motivated by the previous result, the purpose of this article is to study the off-diagonal case. Since the $k$-skeleton maximal operator is bounded on $L^p$, $p\geq 1$, and  is trivially bounded on $L^\infty$ and from $L^1 \to L^\infty$, by the classical Marcinkiewicz interpolation theorem, applied to the bigger (normwise) linear maximal operator (see Definition \ref{linearmaximal}), we can obtain the boundedness from $L^p \to L^q$, $1<p\leq q$. Nevertheless,
   as we mentioned before, we are interested at the rate at which the norm of the $k$-skeleton maximal operator from $L^P$ to $L^q$ increases as the parameter $\delta$ tends to 0, where as usual
 \[
 \norm{M^k_\delta}_{L^p \rightarrow L^q} = \sup_{f \neq 0} \frac{\norm{M^k_\delta f}_{L^q}}{\norm{f}_{L^p}}.
 \] 

In the next, we say that  $f(\delta) \approx g(\delta)$ if there exist two positive constants $c, c'$, not depending on $\delta$, such that 
$
c' g(\delta) \leq f (\delta) \leq c g(\delta)$. 

Our main result is the following:

\begin{theorem} \label{mainresult} Given $1< p \leq q < \infty$  and $\varepsilon>0$, there exist positive constants $C=C(p,q,k,n)$ and $C'=C'(p,q, k , n, \varepsilon)$  such that,
\begin{equation*}\label{cota-regionI}
\begin{split}
 C'\delta^{\frac{k-n}{2np} + \varepsilon} \leq \norm{M^k_\delta}_{L^p\rightarrow L^q} &\leq C \delta^\frac{k-n}{2np} \hspace{1.2cm}   \textup{if}\,\, q\leq q^*p \\
\norm{M^k_\delta}_{L^p\rightarrow L^q} &\approx \delta^{\frac{n}{q}-\frac{n-k}{p}}   \hspace{1cm}  \textup{if} \,\, q > q^*p  %
\end{split}
\end{equation*}

for all $\delta \in (0,1)$ and $q^*=\frac{2n^2}{(n-k)(2n-1)}$.
\end{theorem}

The lower bound follows by applying $M^k_\delta$ to an appropriate function. 
For the upper bounds, we apply a combinatorial method used by Schlag in \cite{Sch96}  for the circular maximal operator in the plane, combined with classical interpolation theorems and an estimate from the combinatorial Lemma \ref{chooseelement}, that we will present in the next section.

For the remaining case, $q>p$, it is straightforward conclude that the $k$-skeleton maximal operator is unbounded. 
 In fact, given $N \in \mathbb{N}$, let $f$ be the characteristic function of an $n$-dimensional cube with side length $N$ and consider another cube, namely $N^*$, with the same center and side length $N-6$. By a simple calculation we obtain $M^k_\delta f(x) \geq 1$, for all $x \in N^*$. Therefore,
\[
\norm{M^k_\delta}_{L^p \rightarrow L^q} \geq (N-6)^{n/q}N^{-n/p},
\]
which grows with $N.$






\section{Preliminaries and notation}\label{sec:preliminaries and notation}
 
Most of the following definitions and results were introduced in \cite{OS18}, for completeness here we give a brief summary. In particular, we  define the linear maximal operator that it turn out be the key to obtain upper bound estimates.

We denote the half-open unit cube by $Q_0$, i.e. $Q_0=[0,1)^n$. Let $0<\delta<1$ such that $1/\delta$ is an integer and consider the grid $Q_0 \cap \delta\Z^n$. We define $Q_0^* := \{x_1,\ldots,x_u\}$, $u = \delta^{-n}$, the centers of the half-open $n$-cubes, $Q_0(1), \ldots, Q_0(u)$, with side length $\delta$ determined by the grid $Q_0 \cap \delta\Z^n$. 

We define the function $\psi: Q_0 \rightarrow Q_0^*$, $x \rightarrow x^*$, where $x^*$ denotes the center of the corresponding half-open $n$-cube with center in $Q_0^*$ and side length $\delta$ containing $x$. 
Observe that $\psi$ is constant over each $Q_0(i)$ and the sets $\psi^{-1}(Q_0(i))$,  $i=1, \ldots, u$, form a Borel partition of $Q_0$.

 We use the letter C to denote positive constants, indicating any parameters they may
depend on by subindices. Their values may change from line to line. For example,
$C_{n}$ denotes a positive function of $n$. 

\medskip

The next combinatorial lemma is crucial in our work. It states,  roughly speaking, that given a finite family of $k$-skeletons, we can extract one face from each skeleton and the overlap among them is controlled.

\begin{lemma} \cite[Lemma 3.2]{OS18} \label{chooseelement}
There is a constant $C_{n,k}<\infty$, depending only on $n,k$, such that the following holds. Let $\{S_k(x_i,r_i)\}_{i=1}^m$ be a finite collection of $k$-skeletons in $\rr^n$. Then it is possible to choose one $k$-face of each skeleton with the following property: If $V$ is an affine $k$-plane which is a translate of a coordinate $k$-plane, then $V$ contains at most
\[
C_{n,k}\displaystyle m^{1-\frac{(n-k)(2n-1)}{2n^2}}
\]
of the chosen $k$-faces.
\end{lemma}

\begin{definition} \label{defn:funcion Phi}
Let $\Gamma_0$ denote the family of all functions  $\rho: Q^*_0  \rightarrow [1,2]\cap \delta\Z$.
 Fix also  $\rho\in\Gamma_0$.  For simplicity, let us write S$_{k,i}= S_k(x_i,r_i)$, where $r_i=\rho(x_i)$ and $1 \leq i \leq u$.

For this family of $k$-skeletons, we define the function $\Phi_{\rho}$:
\begin{equation}\label{eq: family k faces}
\Phi_{\rho} (S_{k}(x_i,r_i)) = \ell^i_k,
\end{equation}
where $\ell_k^i$ denotes the face of $S_{k,i}$ chosen as in Lemma \ref{chooseelement}.

\end{definition}

 The $k$-skeleton maximal operator defined in \eqref{k-skeleton maximal}, unlike most other kinds maximal operators, is not sub-linear. To deal with this inconvenient we introduce a discretized and linearized version of the problem. 

\begin{definition} \label{linearmaximal}  Given a function $\rho\in\Gamma_0$ and $0<\delta <1$, if $f \in L^{1}_{\rm loc}(\rr^n)$ we define the  \textit{$(\rho,k)$-skeleton maximal function} with width $\delta$,
$
{\widetilde{M}}^{k}_{\rho,\delta}f : Q_0 \rightarrow \rr
$ by
\begin{center}
\begin{equation}\label{eq: linear maximal}
{\widetilde{M}}^{k}_{\rho,\delta}f(x)= \displaystyle\frac{1}{|\ell_{x,\delta}|} \int _ {\ell_{x,\delta}} |f(y)| \,dy,
\end{equation}
\end{center}
where $\ell_{x,\delta}$ is a $\delta$-neighborhood of $\ell_{x}:=\Phi_{\rho} (S_{k}(x^*, \rho(x^*)))$.
\end{definition}

\begin{remark}\label{rmk: maximal lineal}
By definition, $\widetilde{M}^k_{\rho,\delta}$ is constant over each set $Q_0(i)$, $i=1,\ldots,u$, and it is completely determined by its values over the set $Q_0^* = \{x_1,\ldots, x_u\}$.

 Let $CQ_0$ denote
the $n$-cube with the same center as $Q_0$ and side length $C$.
Since $r$ is bounded by 2, in the previous definition it is enough to consider functions $f$ supported on $7Q_0$, since for each $x \in Q_0$, we have that  $\ell_{x,\delta} \subset 7Q_0$.
\end{remark}

\medskip

The following result establishes the normwise relation between the $k$-skeleton maximal operator and its linearized version.
\begin{lemma}
 \cite[Lemma 2.6]{OS18} \label{sup} There exists a constant $C_{k,n}>0$ such that  if $0 < \delta <1$,
\[
 \norm{M^k_{\delta}}_{L^{p}\rightarrow L^{q}(Q_0)}\leq C_{k,n} \sup_{\rho\in\Gamma_0} \norm{\widetilde{M}^k_{\rho,3\delta}}_{L^{p}\rightarrow L^{q}(Q_0)}.
\]
\end{lemma}

In consequence, by obtaining $L^p\rightarrow L^q$ estimates on the discrete maximal operator uniformly on $\rho$, we will also obtain $L^p\rightarrow L^q$ bounds for $M^k_\delta$, at least at a local level.

\section{ the weak endpoint estimate}
In the present section we
shall prove a weak-type $(1,q^*)$ estimate, 
for the $(\rho,k)$-skeleton maximal function, the linearized version of $M^k_\delta$.
To achieve this, we will follow some ideas used in \cite{Sch96} to treat the circular maximal operator.

 Although we are interested on $k$-skeletons, first we establish the setting for more general sets.

\medskip 

\subsection{General setting} Let $\Omega \subseteq \rr^n$ be a Borel set and $I$ an index set. For each $x \in \Omega$, consider a family $\mathcal{A}_{x} = \{A_{x,r}\}_{r \in I}$  of sets with positive and finite Lebesgue measure in $\rr^n$ such that $\displaystyle\inf_{r \in I} |A_{x,r}|  \geq \delta^{\tau}$ for some $\tau \in \mathbb{R}$ and $0<\delta<1$. Intuitively, we can think that $\mathcal{A}_{x}$ is a family of sets at $\delta$-scale with measure equal to $\delta^{\tau}$, for all $x \in \Omega$.
 
In addition, suppose there exists a constant $C>0$, not depending on $x$, such that
 \begin{equation}\label{geom-cond-A}
  {\rm diam}\left(\{x\} \cup \bigcup_{r \in I} A_{x,r}\right) \leq C
 \end{equation}
 
and for all $f \in L^1_{\rm loc}(\rr^n)$ consider the maximal operator $T f : \Omega \rightarrow \rr $ defined by
 \[
T f(x) = \sup_{r \in I }\frac{1}{|A_{x,r}|}\int_{A_{x,r}} |f(y)|  \, dy.
\]

\medskip

Let $E \subset \rr^n$ be a set with finite Lebesgue measure, $0 < \lambda \leq 1$ 
, and $\{x_j\}_{j=1}^m$, a maximal $\delta$-separated sequence in 
\begin{equation}\label{eq:set F}
F= \{ x \in \Omega : T\1_{E} (x) >\lambda \}.
\end{equation}
Note that, by \eqref{geom-cond-A}, $F$ is a bounded set. Pick $r_j \in I$ such that
\[
|A_{x_j,r_j} \cap E| > \lambda  |A_{x_j,r_j}|, \quad  \quad 1 \leq j \leq m.
\]

To simplify the notation, in the next write $A_j$ instead of $A_{x_j, r_j}$ and $A^*_j$ instead of $A_{x_j,r_j} \cap E$. Consider the multiplicity function 
\[
\Upsilon= \displaystyle \sum_{j=1}^m \1_{A^*_j}
\]
and define $\mu$ to be the smallest integer for which there exist at least $m/2$ values of $j$ such that 
\[
|\{ x \in A^*_ j : \Upsilon (x) \leq \mu \}| \geq \frac{\lambda}{2} |A_j|.
\]

Observe that
\begin{equation} \label{measureE}
\mu |E| \geq \int_{\{x \in E: \Upsilon(x)\leq \mu\}} \Upsilon \, dx= \displaystyle \sum _{j=1}^m| \{  x \in A^*_j : \Upsilon (x)\leq \mu\}| \geq \frac{\lambda}{2} m \delta^{\tau}.
\end{equation}
 
The following lemma characterizes the estimates on $\mu$ required to obtain restricted weak-type $(p,q)$ estimates for the maximal operator $T$.

Recall that $T$ is said to be 
of weak-type $(p,q)$ with norm $K$ and write
\[
\norm{Tf}_{q,\infty} \leq K \norm{f}_{p},
\]
if for all $f \in L^p(\rr^n)$ and $t>0 \in (0,1]$, 
\begin{equation*}\label{eq: weak type defn}
|\{ x \in \Omega : T f(x) > t \}|^{1/q} \leq K t^{-1} \norm{f}_{L^p}.
\end{equation*}
When the above inequality holds for characteristic functions of arbitrary sets in $\rr^n$ of finite measure, we say that $T$ is of restricted weak-type $(p,q)$.

\begin{lemma}\label{combinatorialmethod}
Let $\alpha\geq 0$ and $\beta < 1$.  There exists a positive constant $C=C(n,q)$ such that, if  $\mu \leq H \lambda^{-\alpha}m^{\beta}$ for every choice of a bounded set $E \subset \mathbb{R}^n$, $0 < \lambda \leq 1$ and $\delta > 0$, then $T$ is of restricted weak-type $(p,q)$ with constant $CH^{1/p} \delta^{-\gamma}$,
 where $p=\alpha+1, q=p(1-\beta)^{-1}$ and $\gamma= \frac{\tau}{p}-\frac{n}{q}$. 
\end{lemma}

\begin{proof}
We need to prove that 
\begin{equation}\label{weakbound}
|\{ x \in \Omega :  T \1_E(x) > \lambda \}|^{1/q} \leq C H^{1/p} \delta^{-\gamma}\lambda^{-1} |E|^{1/p},
\end{equation}
for every set $E \subset \mathbb{R}^n$ of finite measure.

Let $E\subset \mathbb{R}^n$  and $\{x_j\}_{j=1}^{m}$ be a maximal $\delta$-separated sequence in $F$, as in \eqref{eq:set F}. Then,  
\[
|\{ x \in \Omega:  T \1_{E}(x) > \lambda\}| \leq c_n \delta^n m ,
\]
where $c_n$ denotes the measure of the $n$-dimensional ball with radius 1.
In view of (\ref{measureE}), i.e. $|E|\geq \mu^{-1}\frac{\lambda}{2} m \delta^{\tau}$, and by the assumption on $\mu$ we conclude that 
\begin{align*}
 H^{1/p} \delta^{-\gamma}\lambda^{-1} |E|^{1/p} &\geq  H^{1/p} \delta^{-\gamma} \lambda^{-1}(H^{-1}(\lambda/2)^{1+\alpha} m^{1-\beta} \delta^{\tau})^{1/p} \\ &=\frac{1}{2} \delta^{-\gamma + \tau/p} m^{\frac{1-\beta}{p}}  
  \\ &= \frac{1}{2}(m \delta^n)^{1/q}.
\end{align*}
 Therefore, \eqref{weakbound} follows with constant $C=2c_n^{\frac{1}{q}}$. 
\end{proof}

\subsection{The case of $k$-skeletons}

First, we introduce some notation and definitions.
\begin{itemize}
\item Given $n\geq 2$, $0\leq k< n \in \mathbb{N}$, and $\{e_1,\ldots, e_n\}$, the canonical base in $\mathbb{R}^n$, we denote by $\pi_1, \ldots, \pi_{n \choose k}$ the ${n \choose k}$ subspaces of dimension $k$ generated by vectors in the canonical base. From now on, we denote them as $k$-planes. 
  For example, if $n=2$ and $k=1$, $\pi_1$ and $\pi_2$ are the usual coordinates axes in the plane. If $k=0$, the $0$-plane will be the origin. 

\item Let $\rho \in \Gamma_0$ and $\ell^1_k, \ldots, \ell^u_k$ as in (\ref{eq: family k faces}). We say that $\ell^i_k$, $1 \leq i\leq u$, is parallel to $\pi_\omega$ if there exists a point $v \in \rr^n$ such that it is contained in the affine subspace  $V = \pi_\omega + v$.

\item 
For each $\omega=1,\ldots, {n \choose k}$, we define the sets
\[
E_{\pi_w} := \{x_i \in Q_0^* : \ell^k_i \, \text{ is parallel to} \, \pi_w\}.
\]
\end{itemize}
 
\medskip 
 
  Let $E \subset \rr^n$ be a set with finite measure, $0< \lambda\leq 1$ and consider the set
\begin{eqnarray*}
F_{\omega}:= \{ x \in \psi^{-1}(E_{\pi_\omega}) : \widetilde{M}^k_{\rho,\delta} {\1}_{E}(x) >\lambda\},
\end{eqnarray*}
where the index $\omega$ is fixed.
By Remark \ref{rmk: maximal lineal}, $F_\omega$ is the union of those half-open cubes $Q_0(i)$ with center $x_i \in F_\omega$.
Note that it is enough just to consider those sets $E$ such that  $E\cap 7Q_0 \neq \emptyset$, otherwise $F_\omega = \emptyset$.

The set $Q^*_0 \cap F_\omega := \{x_j\}_{j=1}^{m_\omega}$ is a maximal $\delta$-separated set in $F_\omega$ and for each $x_j$, we have
\begin{equation*}\label{eq:family delta separated}
|E \cap \ell^j_{k,\delta}| >\lambda |\ell^{j}_{k,\delta}|.
\end{equation*}
For simplicity, we write $(\ell^j_\delta)^*$ instead of $E \cap \ell^j_{k,\delta}$ and $\ell^j_\delta$ instead of $\ell^j_{k,\delta}$.

\medskip

We define the multiplicity function associated to $F_\omega$ by 
\[
\Upsilon_{\omega} = \displaystyle\sum_{j=1}^{m_\omega} \1_{(\ell^j_\delta)^*} 
\]
and $\mu$, as before, the smallest integer such that there exists at least $m_{w}/2$ values of $j$ such that 
\[
|\{x \in  ({\ell}^j_{\delta})^* : \Upsilon_{\omega}(x) \leq \mu \}| \ge \frac{\lambda}{2} |\ell^j_\delta|.
\]

 Fix $1\leq j \leq m_\omega$ and consider $x \in \ell_\delta^j$. Since $ (\ell^j_\delta)^* \subseteq \ell^j_\delta$ we have that 
\[
\Upsilon_{\omega} \leq \displaystyle\sum_{j=1}^{m_\omega} \1_{\ell^j_\delta} .
\]
The faces $\{\ell_k^1, \ldots, \ell_k^{m_\omega}\}$ were chosen from a family of $k$-skeletons using Lemma \ref{chooseelement} and each one of them belongs to an affine $k$-plane parallel to $\pi_\omega$.
Therefore, by means of the mentioned Lemma, we obtain an estimate 
for the number of faces containing $x$. More precisely, for all $x \in \ell_\delta^j$ we have,
\[
\Upsilon _{\omega}(x) \leq C_{n,k} m_{\omega}^{1-\frac{(n-k)(2n-1)}{2n^2}}.
\]
Since this holds for every $j =1,\ldots, m_\omega$, by the definition of $\mu$, we obtain 
\begin{equation}\label{eq: cota mu}
\mu \leq C_{n,k} m_{\omega}^{1-\frac{(n-k)(2n-1)}{2n^2}}.
\end{equation}

\medskip 

The following lemma provides us the weak-type $(1,q^*)$ estimate mentioned at the beginning of the section. 
 
\begin{lemma}\label{weak bound} 
Given $1<q< \infty$, $\rho \in \Gamma_0$, $0<\delta <1$ and $0 \leq k< n$, there exists a positive constant $C=C(k,n,q)$ such that
\begin{equation*}
\begin{split}
\norm{\widetilde{M}^k_{\rho,\delta} f}_{q,\infty} &\leq C\, \delta^{\frac{k-n}{2np}}\norm{f}_1 \hspace{1.4cm} \textup{if}  \hspace{0.3cm} 1< q \leq q^*\\
\vspace{0.5cm}
\norm{\widetilde{M}^k_{\rho,\delta}f}_{q,\infty} &\leq C\,\delta^{\frac{n}{q}+ k-n}\norm{f}_1 \hspace{ 1 cm} \textup{if}  \hspace{0.3cm} q^*<q < \infty,
\end{split}
\end{equation*}
for every $f \in L^1(7Q_0)$ and $q^*=\frac{2n^2}{(n-k)(2n-1)}$.
\end{lemma}

\begin{proof}
Consider the restricted maximal operator  $\widetilde{M}^k_{\rho,\delta} : \psi^{-1}(E_{\pi_\omega}) \rightarrow \rr$. By \eqref{eq: cota mu} we have \[
 \mu \leq C_{n,k} m_{\omega}^{1-\frac{(n-k)(2n-1)}{2n^2}}.
 \]
  Applying Lemma \ref{combinatorialmethod} with $\alpha=0$, $\beta = 1- \frac{(n-k)(2n-1)}{2n^2}$, $H = C_{n,k}$ and $\tau= n-k$, we obtain 
\begin{equation} \label{weak-type-w}
|\{ x \in \psi^{-1}(E_{\pi_w}) : \widetilde{M}^k_{\rho,\delta}\1_E (x) > \lambda \} |^{1/q^*} \leq C_{k,n} \delta^{\frac{k-n}{2n}} \lambda^{-1} |E|,
\end{equation}
for every set $E \subset \rr^n$ of finite measure. 

Since this holds for every $\omega=1,\ldots, {n\choose k}$, and the sets $\psi^{-1}(E_{\pi_\omega})$ form a Borel partition of $Q_0$, we have
\begin{align*}
| \{ x \in Q_0 : \widetilde{M}^k_{\rho,\delta} \1_{E}(x) > \lambda \} |&\leq \sum_{\omega=1}^{{n \choose k}} \lvert \{  x \in \psi^{-1}(E_{\pi_\omega}): \widetilde{M}^k_{\rho,\delta} \1_{E}(x) > \lambda\} \rvert\\ &\leq {n \choose k} (C_{n,k}\delta^{\frac{k-n}{2n}} \lambda^{-1}|E|)^{q^*}
\end{align*}
and therefore, we can conclude that $\widetilde{M}^k_{\rho,\delta}$ is of restricted weak type $(1,q^*)$. 

If $1<q<q^*$, 
\begin{align*}
 |\{x \in Q_0: \widetilde{M}^k_{\rho,\delta}\1_{E} > \lambda \} |^{1/q} &\leq |\{x \in Q_0: \widetilde{M}^k_{\rho,\delta}\1_{E }> \lambda \} |^{1/q^*}   \\&\leq C_{k,n} \lambda^{-1} \delta^{\frac{k-n}{2n}} |E|,
 \end{align*}
and $\widetilde{M}^k_{\rho,\delta}$ is of restricted weak type $(1,q)$.

For the remaining case $q> q^*$, take a constant $v>0$ such that $\frac{1}{q} = \frac{1}{q^*} - v$. Trivially by \eqref{eq: cota mu},
\[
\mu \leq C_{n,k} m_{\omega}^{1-\frac{(n-k)(2n-1)}{2n^2} +v}.
\]
Invoking Lemma \ref{combinatorialmethod} with $\beta= 1-\frac{(n-k)(2n-1)}{2n^2} +v$, $\alpha=0$ and $K= C_{n,k}$ we obtain,
\[
|\{ x \in \psi^{-1}(E_{\pi_w}) : \widetilde{M}^k_{\rho,\delta}\1_E (x) > \lambda \} |^{1/q} \leq C_{k,n} \delta^{\frac{n}{q}-(n-k)} \lambda^{-1} |E|.
\]
 Therefore,
\[
| \{ x \in Q_0 : \widetilde{M}^k_{\rho,\delta} \1_{E}(x) > \lambda\} |^{1/q} \leq C_{k,n}\delta^{\frac{n}{q}+k-n} \lambda^{-1}|E|,
\]
which implies that $\widetilde{M}^k_{\rho,\delta}$ is of restricted weak type $(1,q)$ with $q>q^*$.

Finally, by \cite[Theorem 5.5.3]{BS88}, we can conclude that $\widetilde{M}^k_{\rho, \delta}$ is of weak-type $(1,q)$.
\end{proof}

\section{Proof of the main theorem}
In order to prove the upper bounds in Theorem \ref{mainresult}, we first establish the following result.
\begin{proposition} \label{p,q linear estimates} For every $ 1 < p \leq q $, $0< \delta<1$ and $\rho \in \Gamma_0$, there exist positive constants $C$ and $C'$  depending on $k,n,p,q$ such that, 
\begin{equation*}
\begin{split}
\norm{\widetilde{M}^k_{\rho,\delta}}_{L^p \rightarrow L^q(Q_0)} &\leq C\, \delta^{\frac{k-n}{2np}}  \hspace{1cm} \textup{if}\, \, q\leq q^*p \\
\norm{\widetilde{M}^k_{\rho,\delta}}_{L^p \rightarrow L^q(Q_0)} &\leq C' \,\delta^{\frac{n}{q}+\frac{k-n}{p}} \hspace{0.5cm} \textup{if} \,\, q > q^*p.
\end{split}
\end{equation*}

\end{proposition}

\begin{proof}  Given $(p,q)$ with $p \leq q$, by Lemma \ref{weak bound}, the trivial bound
\begin{equation*}
\norm{\widetilde{M}^k_{\rho,\delta} f}_{L^{\infty} (Q_0)} \leq \norm{f}_{L^{\infty} (7Q_0)},
\end{equation*}
and the Marcinkiewicz interpolation Theorem (See e.g.\cite[Theorem 4.4.13 ]{BS88}) we have the desire result.
\end{proof}

\medskip

For each $z=(z_1,\ldots,z_n) \in \Z^n$ we denote
\[
Q_z = [z_1,z_1+1)\times\cdots\times [z_n,z_n+1).
\]
By translation invariance, Proposition \ref{p,q linear estimates}  continues to hold if we replace $Q_0$ by $Q_z$. 

\begin{proof}[ Proof of Theorem \ref{mainresult}]: 
 In both cases, the upper bounds follow from Lemma \ref{sup}, Proposition \ref{p,q linear estimates} and the facts that  
 $\rr^n= \cup_{z}Q_z$ and 
$ \norm{\sum_{z \in \Z^n} \1_{7Q_z}}_{\infty}$ is finite.

If  $q> q^*p$, to obtain the lower bound, we consider  $f=\1_{B_{6\delta}}$, where $B$ is the $k$-skeleton of an $n$-cube with side length one and center in some point $x_0 \in \rr^n$. 
It is easy to see that $M^{k}_\delta f(x) \ge 1$  for all $x$  in a $\delta$-neighborhood of $x_0$  and $\norm{M^k_\delta f}_{L^q} \ge \delta^{n/q}$. 
Since $\norm{f}_{L^p} = |B_{6\delta}|^{n-k/p} \approx \delta^{n-k/p}$, we have \[ 
\norm{M^k_\delta}_{L^p \rightarrow L^q} \geq c_{n,k} \delta^{\frac{n}{q} + \frac{k-n}{p}}.
\] 

For the remaining lower bound in the case $q\leq q^*p$, see \cite[Proposition 2.2]{OS18}. 

\end{proof}

\section*{Acknowledgments}
I thank my supervisor Pablo Shmerkin
for his guidance and for his extremely valuable comments and suggestions.


\end{document}